\documentclass[leqno,11pt]{article}
\usepackage{amsmath,amssymb}
\usepackage{color}
\usepackage{epsfig}
\usepackage{graphicx}
\usepackage[utf8]{inputenc}
\usepackage[T1]{fontenc}
\usepackage{url}

\usepackage{dsfont}
\def\R{{\mathbb{R}}}

\newcounter{theorem}

\numberwithin{equation}{section}
\numberwithin{theorem}{section}
\newtheorem{theorem}{Theorem}[section]

 \newtheorem{cor}[theorem]{Corollary}
 
 \newtheorem{lem}[theorem]{Lemma}
 
 \newtheorem{prop}[theorem]{Proposition}

\newcommand{\qed}{\hfill$\square$\vspace{0.3cm}}

\begin{document}

\title{\textbf{Remarks on a nonlinear nonlocal operator \\ in Orlicz spaces}}
\author{
Ernesto Correa, Arturo de Pablo }
\date{ }
\maketitle

\begin{abstract}
We study integral operators $\mathcal{L}u(x)=\int_{\mathbb{R^N}}\psi(u(x)-u(y))J(x-y)\,dy$ of the type of the fractional $p$-Laplacian operator, and the properties of the corresponding Orlicz and Sobolev-Orlicz spaces. In particular we show a Poincar\'e inequality and a Sobolev inequality, depending on the singularity at the origin of the kernel $J$ considered, which may be very weak. Both inequalities lead to compact inclusions. We then use those properties to study the associated elliptic problem $\mathcal{L}u=f$ in a bounded domain $\Omega$, and boundary condition $u\equiv0$ on $\Omega^c$; both cases $f=f(x)$ and $f=f(u)$ are considred, including the generalized eigenvalue problem $f(u)=\lambda\psi(u)$.
\end{abstract}

\vskip 1cm

%
%\noindent $\underline{ \;\;\;\;\;\;\;\;\;\;\;\;\;\;\;\;\;\;\;\; } $ \newline
\noindent{\makebox[1in]\hrulefill}\newline

%

%%%%%%%%%%%%%%%%%%%%%%%%%%%%%%%%%%%%%%%%%%%%%%%%
\section{Introduction}\label{sect-introduction}
\setcounter{equation}{0}

The aim of this paper is to study the properties of the nonlinear nonlocal operator
\begin{equation}
\label{operator}
\mathcal{L}u(x)=\mathcal{L}^{J,\psi}u(x)\equiv\int_{\mathbb{R^N}}\psi(u(x)-u(y))J(x-y)\,dy,
\end{equation}
where $\psi:\mathbb{R}\to\mathbb{R}$ is a nondecreasing, continuous, unbounded odd function,
and $J:\mathbb{R}^N\to\mathbb{R}^+$ is a measurable function satisfying
\begin{equation}
\label{kernel0}
\tag{$\text{\rm H}_0$}
\left\{
\begin{array}{l}
J(z)> 0, \quad J(z)=J(-z),\quad J\notin L^1(B_1),\\[10pt]
\displaystyle
\int_{\mathbb{R}^N}\min(1,| z|^{q_0})J(z)\,dz<\infty,\quad \text{for some } q_0>0.
\end{array}
\right.
\end{equation}
$B_r$ denotes the ball $B_r=\{z\in\mathbb{R}^N\,:\,|z|<r\}$. This set of hypotheses is assumed throughout the paper without further mention. We also denote $q_*=\inf\{q_0>0\,:\eqref{kernel0} \text{ holds }\}$, which measures in some sense the differential character of the operator.

The power case $\psi(s)=k|s|^{p-2}s$ for some $p>1$, $J(z)=c|z|^{-N-\sigma p/2}$ for some $0<\sigma<2$, is  known as  the $\sigma$--fractional $p$--Laplacian operator. We want to consider here general functions $\psi$  and $J$ more than just powers, so we are led to study some Orlicz and Sobolev-Orlicz spaces, see below, which makes the study nontrivial. On the other hand, we are also interested in the limit case of integrability, which in our context means $q_*=0$, that is, the singularity of the kernel can be weaker than that of any fractional Laplacian or $p$--Laplacian. Some of the results also hold for more general kernels, $J=J(x,y)$, satisfying only a lower estimate $J(x,y)\ge J_0(x-y)$, with $J_0$ in the above hypotheses, but we prefer to keep the proofs in a simpler way.

For problems including  operators like \eqref{operator}, in particular the fractional $p$--Laplacian, and the motivations for their study we refer to~\cite{Caffarelli}.

\subsection{The associated Orlicz spaces}\label{sect-orlicz}

Formula \eqref{operator} makes sense pointwise for regular functions with some extra restriction on the nonlinearity $\psi$ and the kernel $J$, see Section~\ref{sect-prelim}. In order to define the operator $\mathcal{L}$ in weak sense we consider the \emph{nonlocal nonlinear interaction energy} (linear in the second variable)
\begin{equation}\label{E}
\mathcal{E}(u;\varphi)=\frac12\iint_{\mathbb{R}^{2N}}\psi(u(x)-u(y))(\varphi(x)-\varphi(y))J(x-y)\,dxdy,
\end{equation}
and we put
$$
\langle\mathcal{L}u,\varphi\rangle=\mathcal{E}(u;\varphi).
$$
Clearly, by the symmetry properties of $\psi$ and $J$ we have $\mathcal{E}(u;\varphi)=\displaystyle\int_{\mathbb{R}^N} \mathcal{L}u \varphi$ for regular functions. But the above allows to define $\mathcal{L}$ also for functions in a Sobolev type space. To this end we define the functionals
\begin{equation}
  \label{energyF}
  F(u)=\int_{\mathbb{R}^N}\Psi(u(x))\,dx ,
\end{equation}
\begin{equation}
  \label{energyE}
  E(u)=\frac12\iint_{\mathbb{R}^{2N}}\Psi(u(x)-u(y))J(x-y)\,dxdy,
\end{equation}
with $\Psi'=\psi$.
The properties of $\psi$ imply that $\Psi$ is an strict Young function, so we can consider the Orlicz spaces
\begin{equation}
  \label{orlicz}
  L^\Psi(\mathbb{R}^N)=\{u\,:\,\mathbb{R}^N\to\mathbb{R},\;F(u)<\infty\},
\end{equation}
\begin{equation}\label{W_J}
{W}^{J,\Psi}(\mathbb{R}^N)=\left\{u\in L^\Psi(\mathbb{R}^N),\,  E(u)<\infty\right\}.
\end{equation}
Observe that in general $\mathcal{E}(u;u)\neq c E(u)$ for any constant $c>0$, the equality being true  only in the power case  $\psi(u)=k|u|^{p-2}u$, and then $c=p$. What we have  is that $\mathcal{E}$ is the Euler-Lagrange operator associated to the functional $E$, that is,
\begin{equation*}\label{weakform}
  \langle E'(u),\varphi\rangle=\mathcal{E}(u;\varphi )
\end{equation*}
for every $u,\varphi\in {W}^{J,\Psi}(\mathbb{R}^N)$.

The above spaces do not have good properties unless we impose some conditions on the nonlinearity $\Psi$. The simplest case is when
\begin{equation}
  \label{psi-plap}
  c_1s^{p-1}\le \Psi'(s)\le c_2s^{p-1},\qquad s>0,\quad p>1,
\end{equation}
so that the space $L^\Psi(\mathbb{R}^N)$ coincides with $L^p(\mathbb{R}^N)$, and the Sobolev space ${W}^{J,\Psi}(\mathbb{R}^N)$ is denoted by ${W}^{J,p}(\mathbb{R}^N)$. But we are interested in more general functions. Thus we consider the set, for some $p\ge q>1$,
\begin{equation}\label{set-psi}
\begin{array}{rl}
\displaystyle\Gamma_{p,q}= \Big\{\Psi:\mathbb{R} \to \mathbb{R}^+, &\text{convex, symmetric, satisfying }\Psi(0)=0,\,\Psi(1)=1,\\ [3mm]
&\displaystyle  q\le\frac{s\Psi '(s)}{\Psi(s)} \leq p\quad\forall s\neq0\Big\}.
\end{array}
\end{equation}
The condition $\Psi(1)=1$ is for normalization purposes and simplifies some expression. We thus deal with functions that lie between two powers, for instance a sum of powers, but we also allow for perturbation of powers like $\Psi(s)=c|s|^p|\log(1+s)|^r$, $\min\{p,p+r\}>1$. The first property deduced from~\eqref{set-psi} is the relation between the interaction energy $\mathcal{E}$ and the functional $E$,
\begin{equation}
  \label{equiv-Ee}
  qE(u)\le \mathcal{E}(u;u)\le pE(u).
\end{equation}

Our main interest lies in studying the properties of the spaces~\eqref{energyF} and~\eqref{energyE} for nonlinearities $\Psi$ in the class $\Gamma_{p,q}$. In particular  we have that $L^\Psi(\mathbb{R}^N)$ and ${W}^{J,\Psi}(\mathbb{R}^N) $ are reflexive Banach spaces, with norms defined, for instance, in \eqref{luxemb} and \eqref{luxemb-W}.
On the other hand, if $q> q_*$,  see \eqref{kernel0}, then the functional $E(u)$ is well defined and finite for functions satisfying $F(\nabla u)<\infty$, see Proposition~\ref{prop-W1p}. This means the inclusion $ W^{1,\Psi}(\mathbb{R}^N)\subset{W}^{J,\Psi}(\mathbb{R}^N) $, the former being the standard Orlicz-Sobolev space of functions in $L^\Psi(\mathbb{R}^N)$ with gradient also in $L^\Psi(\mathbb{R}^N)$.

When dealing with  problems defined in bounded domains $\Omega\subset\mathbb{R}^N$, since the Dirichlet conditions must be prescribed in the complement $\Omega^c\equiv\mathbb{R}^N\setminus\Omega$, instead of just on the boundary, precisely by the nonlocal character of the operator, it is convenient to consider the space
\begin{equation*}\label{H_J}
{W}_0^{J,\Psi}(\Omega)=\left\{u\in {W}^{J,\Psi}(\mathbb{R}^N),\, u\equiv0 \text{ in } \Omega^c\right\}.
\end{equation*}

Without imposing any singularity condition on the kernel $J$ at the origin, besides of course being not integrable, that is $q_*$ may be zero, we show that a Poincar\'e inequality $E(u)\ge c F(u)$ holds, so that we have the embedding
\begin{equation}\label{poincare}
{W}_0^{J,\Psi}(\Omega)\subset L^\Psi(\Omega).
\end{equation}
Observe that if $J$ were integrable then ${W}_0^{J,\Psi}(\Omega)\equiv L^\Psi(\Omega)$.

Assuming now $q_*>0$ (and some monotonicity near the origin, \eqref{regular-v})  we obtain a better result, namely a Sobolev embedding
\begin{equation}\label{sobolev}
{W}_0^{J,\Psi}(\Omega)\subset L^{\Psi^r}(\Omega),\qquad 1\le r< r^*\equiv\begin{cases} \frac N{N-q_*}&\text{ if } q_*<N, \\ \infty& \text{ if } q_*\ge N,
\end{cases}
\end{equation}
which is compact. The borderline $r=r^*$ when $q_*<N$ produces also an embedding ${W}_0^{J,\Psi}(\Omega)\subset L^{\Psi^{r^*}}(\Omega)$, provided $J(z)\ge c|z|^{-N-q_*}$ near the origin, but without compactness. In the limit case $q_*=0$, which would give $r^*=1$ in~\eqref{sobolev}, we do obtain compactness of the inclusion~\eqref{poincare} by assuming a minimum of singularity on the kernel, the extra condition $\lim\limits_{|z|\to0^+}|z|^{N}J(z)=\infty$. See Theorems~\ref{th-Sobolev}--\ref{th-compact2}.

\subsection{Elliptic problems}\label{sect-elliptic}

With this machinery we next study the problem
\begin{equation}\label{problem}
\begin{cases}
\mathcal{L}u=f,& \mbox{in }\Omega,\\ u=0,& \mbox{in }\Omega^c.
\end{cases}
\end{equation}
This problem must be considered in weak sense with the aid of the interaction energy $\mathcal{E}$, that is, any solution $u$ satisfies
\begin{equation}\label{weaksol}
\mathcal{E}(u;\varphi)=\int_\Omega f\varphi, \qquad\forall \;\varphi\in{W}^{J,\Psi}_0(\Omega).
\end{equation}
We study first the case $f=f(x)$ in an appropriate space. We obtain existence and uniqueness of a solution, see Theorem~\ref{th-existence-linear}. We also show some integrability properties of the solution in terms of the data $f$ when $\psi$ is restricted to the power-like case~\eqref{psi-plap}, see Theorems~\ref{th-Lp1} and~\ref{th-Lp2}. In particular the solution is bounded provided $f\in L^m(\Omega)$ with $m>N/q_*$ if $q_*>0$, see Theorem~\ref{Linf}. For the corresponding results in the case of the fractional $p$--Laplacian see~\cite{BarriosPeralVita}.

We then pass to study the case $f=f(u)$ in problem \eqref{problem}. In Theorem~\ref{th-sublinear} we prove existence of a nonnegative nontrivial solution in the lower range, which roughly speaking in the power-like case $\psi(s)\sim s^{p-1}$,  $f(t)\sim ct^{m-1}$, means $m<p$. The intermediate range $p<m<m^*=\frac{Np}{N-q_*}$, below the Sobolev exponent, is studied in Theorem~\ref{th-subcritical} using the Mountain Pass Theorem. We also use a Pohozaev inequality  in order to show nonexistence, in the exact power, case for supercritical powers $m>m^{**}=\frac{Np}{N-\delta}$, $\delta>0$ being a constant depending on the kernel $J$, see Corollary~\ref{cor-p^*}. We must remark that all the conditions on the reaction $f$ are very involved in terms on $\psi$, and are not as clean as suggested by the above, see the precise conditions~\eqref{sub1} and \eqref{rho}.  We refer to~\cite{Iannizzotto-etal} for the study of  nonlinear problems like the above, even with more general reactions, for the fractional $p$--Laplacian case.

We finally are interested in the limit case $m=p$, which corresponds to the generalized eigenvalue problem
\begin{equation*}\label{eigenvalues}
\begin{cases}
\mathcal{L}u=\lambda\psi(u),& \mbox{in }\Omega,\\ u=0,& \mbox{in }\Omega^c.
\end{cases}
\end{equation*}
We prove that there exists a first positive eigenvalue and a first positive  eigenfunction, which is bounded provided $q_*>0$, Theorem~\ref{Principal_eigen}. The fractional $p$--eigenvalues have been studied in~\cite{Lindgren-Lindqvist} and~\cite{Franzina-Palatucci}.

\subsection{Organization of the paper}\label{sect-organ}

We begin with a preliminary Section~\ref{sect-prelim} where we study the properties of the Orlicz spaces $L^\Psi(\mathbb{R}^N)$ and ${W}^{J,\Psi}(\mathbb{R}^N)$ by means of some inequalities satisfied by the nonlinearity $\Psi$ and the functionals $F$ and $E$. Section~\ref{sect-inclusions} shows the Sobolev inclusions of the spaces ${W}^{J,\Psi}_0(\Omega)$. Finally Sections~\ref{sect-linear}--\ref{sect-eigen} are devoted to the study of problem~\eqref{problem} for the different reactions commented upon before. In what follows the letters $c$ or $c_i$ will denote some constants that do not depend on the relevant quantities, and may change from line to line.

\section{Preliminaries}\label{sect-prelim}
\setcounter{equation}{0}

In this section we study in detail the properties of the Orlicz spaces $L^\Psi(\mathbb{R}^N)$ and ${W}^{J,\Psi}(\mathbb{R}^N)$ defined in \eqref{orlicz} and \eqref{W_J}, and the corresponding spaces in a bounded domain $\Omega$. We refer to \cite{Rao-Ren} for instance for the general theory of Orlicz spaces. We begin by studying the Young functions in the set $\Gamma_{p,q}$. First observe that $\Psi\in \Gamma_{p,q}$, $p\ge q\ge0$, implies
$$
\min\{|s|^p,\,|s|^q\}\le \Psi(s)\le \max\{|s|^p,\,|s|^q\}.
$$

Associated to any given positive function $g:\mathbb{R}^+\to\mathbb{R}^+$  we consider its characteristic functions, for $s>0$,
\begin{equation*}
  \gamma_g^-(s)=\inf_{x>0}\frac{g(sx)}{g(x)},\qquad
  \gamma_g^+(s)=\sup_{x>0}\frac{g(sx)}{g(x)}.
\end{equation*}
These are nondecreasing functions that satisfy
\begin{lem}\label{gammas} For any $\Psi\in\Gamma_{p,q}$, $p\ge q>0$,
$$
\begin{array}{c}
\min\{s^p,\,s^q\}\le \gamma_\Psi^-(s)\le\gamma_\Psi^+(s)\le \max\{s^p,\,s^q\} \\ [3mm]
\dfrac qp\,\dfrac{\gamma_\Psi^-(s)}s \le \gamma_{\Psi'}^-(s)\le\gamma_{\Psi'}^+(s)\le \dfrac pq\,\dfrac{\gamma_\Psi^+(s)}s .
\end{array}
$$
\end{lem}
\begin{proof}
If $s> 1$ we have that
    $$
    \log\left(\frac{\Psi(sx)}{\Psi(x)}\right)=\int_{x}^{sx}\frac{\Psi '(t)}{\Psi (t)}\,dt\leq p \int_{x}^{sx}\frac{1}{t}\,dt=p\log s,
    $$
    and thus $\Psi(sx)\leq s^p \Psi(x)$. The other estimates for $\Psi$ are analogous. The inequalities for $\Psi'$ are deduced from the definition of $\Gamma_{p,q}$.
\end{proof}
\qed

The complementary function $\Phi$ of a Young function $\Psi$ is defined such that $(\Phi')^{-1}=\Psi'$. If we normalize it to satisfy $\Phi(1)=1$ we have, for every $p\ge q>1$ (\cite[Corollary 1.1.3]{Rao-Ren})
$$
\Psi\in \Gamma_{p,q}\;\Leftrightarrow\;\Phi\in \Gamma_{q',p'},\qquad p'=\frac p{p-1},\;q'=\frac q{q-1}.
$$
These two functions satisfy the Young inequality
$$
ab\le \Psi(a)+\Phi(b),\qquad a,\,b\in\mathbb{R},
$$
and equality holds only if $b=\Psi'(|a|)\text{sign}\,a$. From this point on we always assume $q>1$.

Let $\Psi\in \Gamma_{p,q}$ be fixed and consider the corresponding Orlicz space $L^\Psi(\mathbb{R}^N)$. It is a linear space that satisfies
$$
L^p(\mathbb{R}^N)\cap L^q(\mathbb{R}^N)\subset L^\Psi(\mathbb{R}^N)\subset L^p(\mathbb{R}^N)+L^q(\mathbb{R}^N),
$$
and in the case of bounded domains
$$
L^p(\Omega)\subset L^\Psi(\Omega)\subset L^q(\Omega).
$$

Also it is a Banach space with norm, called \emph{Luxemburg norm},
\begin{equation*}
  \label{luxemb}
  \|u\|_{L^\Psi}=\inf\{k>0\,:\,F(u/k)\le1\}.
\end{equation*}
We recall that other equivalent norms are also used in the literature. The following result allows us to use $F(u)$ instead of $\|u\|_{L^\Psi}$ in most calculations.
\begin{lem}
  \label{lem-equiv-norms}
  \begin{equation}\label{equiv-norms}
  \gamma_\Psi^-(\|u\|_{L^\Psi})\le F(u)\le\gamma_\Psi^+(\|u\|_{L^\Psi}).
  \end{equation}
\end{lem}
\begin{proof} Let $a=\|u\|_{L^\Psi}$.  We clearly have $F\left(u/a\right)\le1$. Then
$$
F(u)=\int_{\mathbb{R}^N}\Psi(u(x))\,dx \le\gamma_\Psi^+(a)\int_{\mathbb{R}^N}\Psi\left(\frac{u(x)}a\right)\,dx\le\gamma_\Psi^+(a).
$$
On the other hand, for every $\varepsilon>0$ we have $F\left(u/(a+\varepsilon)\right)>1$, so that

$$
F(u)\ge\gamma_\Psi^-(a+\varepsilon)\int_{\mathbb{R}^N}\Psi\left(\frac{u(x)}{a+\varepsilon}\right)\,dx\ge\gamma_\Psi^-(a+\varepsilon).
$$
\end{proof}\qed

The dual space of $L^\Psi(\mathbb{R}^N)$ is $L^\Phi(\mathbb{R}^N)$, where $\Phi$ is the complementary function, and thus they are both reflexive Banach spaces.

We also consider the Sobolev type space ${W}^{J,\Psi}(\mathbb{R}^N)$. In the same way as before it is a Banach space with norm
\begin{equation*}
  \label{luxemb-W}
  \|u\|_{W^{J,\Psi}}= \|u\|_{L^\Psi}+[u]_{W^{J,\Psi}}\equiv \|u\|_{L^\Psi}+\inf\{k>0\,:\,E(u/k)\le1\}.
\end{equation*}
The second term is a kind of Gagliardo seminorm in the context of Young functions. For this seminorm an analogous property as that of Lemma~\ref{lem-equiv-norms} also holds,
 \begin{equation}\label{equiv-norms2}
  \gamma_\Psi^-([u]_{W^{J,\Psi}})\le E(u)\le\gamma_\Psi^+([u]_{W^{J,\Psi}}).
  \end{equation}

In order to show that this space is reflexive as well we consider the weighted space
$$
L^\Psi(\mathbb{R}^{2N},J)=\left\{w:\mathbb{R}^{2N}\to\mathbb{R},\,  \iint_{\mathbb{R}^{2N}}\Psi(w(x,y))J(x-y)\,dxdy<\infty\right\}
$$
and put
$M=L^\Psi(\mathbb{R}^N)\times L^\Psi(\mathbb{R}^{2N},J)$.
Clearly, the product space $M$
is reflexive. The operator $T : {W}^{J,\Psi}(\mathbb{R}^N)\, \to\, M$ defined by $Tu = [u,w]$, where $w(x,y)=u(x)-u(y)$, is an isometry. Since ${W}^{J,\Psi}(\mathbb{R}^N)$ is a Banach space, $T({W}^{J,\Psi}(\mathbb{R}^N))$ is a closed subspace
of $M$. It follows that $T({W}^{J,\Psi}(\mathbb{R}^N))$ is reflexive (see \cite[Proposition 3.20]{Brezis}), and consequently
${W}^{J,\Psi}(\mathbb{R}^N)$ is also reflexive.

We now take a look at the properties of the space ${W}^{J,\Psi}(\mathbb{R}^N)$ in terms of the properties of the kernel $J$, in particular its singularity at the origin, which is reflected in the exponent $q_*$, see~\eqref{kernel0}.
\begin{prop}
  \label{prop-W1p} If  $\Psi\in\Gamma_{p,q}$ with $p\ge q> q_*$ then
  $$ W^{1,\Psi}(\mathbb{R}^N)\subset{W}^{J,\Psi}(\mathbb{R}^N)$$
  and moreover
\begin{equation}
  \label{EFgrad}
  E(u)\le c(F(u)+F(\nabla u)).
\end{equation}
\end{prop}
\begin{proof}
We decompose the integral
$$
\begin{array}{rl}
E(u)&\displaystyle=\frac12\int_{\mathbb{R}^N}\left(\int_{|z|<1}\Psi(u(x)-u(x+z))J(z)\,dz\right. \\ [3mm]&\displaystyle\hspace{2cm}\left.+
\int_{|z|>1}\Psi(u(x)-u(x+z))J(z)\,dz\right)dx=\frac12(I_1+I_2).
\end{array}
$$
The far away integral is easy to estimate
$$
I_2\le 2\int_{\mathbb{R}^N}\Psi(u(x)) dx \int_{|z|>1}J(z)\,dz= cF(u).
$$
As to the inner integral, we have
$$
\begin{array}{rl}
I_1&\displaystyle\le
\int_{\mathbb{R}^N}\int_{|z|<1}\Psi\left(\frac{u(x)-u(x+z)}{|z|}\right)\gamma^+(|z|)J(z)\,dz dx \\ [4mm]
&\displaystyle\le
\int_{\mathbb{R}^N}\int_{|z|<1}\Psi\left(\int_0^1|\nabla u(x+tz)|\,dt\right)\gamma^+(|z|)J(z)\,dz dx \\ [4mm]
&\displaystyle\le
\int_{\mathbb{R}^N}\int_{|z|<1}\int_0^1\Psi\left(|\nabla u(x+tz)|\right)\,dt\gamma^+(|z|)J(z)\,dz dx
\\ [4mm]
&\displaystyle\le
\int_{|z|<1}\int_0^1\int_{\mathbb{R}^N}\Psi\left(|\nabla u(x+tz)|\right)\,dx\,dt\gamma^+(|z|)J(z)\,dz
\\ [4mm]
&\displaystyle= \int_{\mathbb{R}^N}\Psi\left(|\nabla
u(x)|\right)\,dx\int_{|z|<1}\gamma^+(|z|)J(z)\,dz=c F\left(|\nabla
u|\right),
\end{array}
$$
since $\gamma^+(|z|)=\gamma_\Psi^+(|z|)\le |z|^{q}$ in the set $\{|z|<1\}$, and using hypothesis \eqref{kernel0}.
\qed\end{proof}

If the kernel $J$ behaves like that of the fractional Laplacian
\begin{equation}
  \label{lapla}
  c_1|z|^{-N-\alpha}\le J(z)\le c_2|z|^{-N-\alpha},
\end{equation}
then we also have the following interpolation estimate
\begin{prop}
  \label{prop-interpolate-lapla} If  $J$ satisfies \eqref{lapla} for some $\alpha>0$, and $\Psi\in\Gamma_{p,q}$ with $p\ge q> \alpha$ then
\begin{equation}
  \label{EFinterpolate2}
  E(u)\le cF(u)\min\left\{\left(\frac{F(\nabla u)}{F(u)}\right)^{\alpha/p},\,\left(\frac{F(\nabla u)}{F(u)}\right)^{\alpha/q}\right\}.
\end{equation}
\end{prop}
\begin{proof}
We apply inequality \eqref{EFgrad} to the rescaled function $u_\lambda(x)=u(\lambda x)$, thus getting
$$
E(u)\le \lambda^{N-\alpha}E(u_\lambda)\le\lambda^{N-\alpha}c(F(u_\lambda)+F(\nabla u_\lambda))\le\lambda^{N-\alpha}c(\lambda^{-N}F(u)+\gamma^+_\Psi(\lambda)\lambda^{-N}F(\nabla u)).
$$
Minimizing the right-hand side in $\lambda$ we obtain the values
$$
\lambda=\left(\frac{\alpha F(u)}{(p-\alpha)F(\nabla u)}\right)^{1/p}\qquad \text{ or } \qquad \lambda=\left(\frac{\alpha F(u)}{(q-\alpha)F(\nabla u)}\right)^{1/q},
$$
depending on the the inner function being bigger or smaller that one. From this we deduce \eqref{EFinterpolate2}.
\qed\end{proof}

In the power-like case we obtain from the above the well-known interpolation result.
\begin{cor}
  If  $J$ satisfies \eqref{lapla} for some $\alpha>0$, and $\psi$ satisfies \eqref{psi-plap} with $p> \alpha$ then
\begin{equation*}
  \label{EFinterpolate}
  E(u)\le cF^{1-\alpha/p}(u)F^{\alpha/p}(\nabla u),
\end{equation*}
or which is the same
\begin{equation*}
  \label{EFinterpolate3}
  \|u\|_{W^{\alpha/2,p}}\le c\|u\|_p^{1-\alpha/p}\,\|\nabla u\|_p^{\alpha/p}.
\end{equation*}\end{cor}

We now turn our attention to the operator $\mathcal{L}$. The pointwise expression \eqref{operator} does not always have a meaning. Let us look at some easy situations where $\mathcal{L}u$ is well defined.

We may take, for instance,  $\Psi''$ nondecreasing and $u\in C_0^2(\mathbb{R}^N)$. Another less trivial example is $q> q_*+1$ and $u\in C^\alpha(\mathbb{R}^N)\cap L^\infty(\mathbb{R}^N)$ with $\dfrac{q_*}{q-1}< \alpha<1$, so that
$$
|\mathcal{L}u(x)|\le\|u\|_\infty\int_{|x-y|>1}J(x-y)\,dy+\int_{|x-y|<1}|x-y|^{(q-1)\alpha}J(x-y)\,dy<\infty.
$$

We now show some useful inequalities. The first one is  a Kato type inequality, that is, the result of applying the operator $\mathcal{L}$ to a convex function of $u$. We refer to \cite{Kato} and  \cite{CoCo}, respectively, for the well-known inequalities
$$
-\Delta|u|\le\text{sign}(u)(-\Delta) u,\qquad(-\Delta)^{\sigma/2}(u^2)\le2u(-\Delta)^{\sigma/2} u.
$$

 \begin{prop}\label{Kato}
    If $A$ is a positive convex function and $\mathcal{L}u$ is well defined, then $\mathcal{L}(A(u))$ is also well defined and
    \begin{equation*}\label{kato-plap}
\mathcal{L}(A(u))\le \gamma^+_\psi(A'(u))\mathcal{L}u.
    \end{equation*}
 \end{prop}

\begin{proof}
We just observe that since $A$ is convex and $\psi$ is nondecreasing, we have
$$
\psi(A(u(x))-A(u(y)))\le \psi(A'(u(x))(u(x)-u(y)))\le \gamma_\psi^+(A'(u(x)))\psi(u(x)-u(y)).
$$
Now integrate with respect to $J(x-y)\,dy$ to get the result.
    \end{proof}\qed

As a Corollary we obtain an integral version of the Kato inequality, useful in the applications.
\begin{cor}
  Assume $G\ge \gamma_\psi^+(A')A$. Then
$$
\mathcal{E}(u,G(u))\ge qE(A(u)).
$$
\end{cor}

Of later use are also the following two inequalities
\begin{equation}\label{kato-simple}
  \mathcal{E}(u,u^+)\ge \mathcal{E}(u^+,u^+),\qquad E(u)\ge E(|u|),
\end{equation}
whose proof is immediate just looking at the signs of the corresponding functions.

Related to those inequalities is  the well known Stroock-Varopoulos inequality,  see \cite{Varopoulos} for the linear case $\Psi(s)=|s|^2$ and $J(z)=|z|^{-N-\sigma}$ for some $0<\sigma<2$, and \cite{BrandledePablo} for general L\'evy kernels $J$. It is of the type of the integral Kato inequality, but the functions for which it holds is different. In the case of powers they coincide but for the coefficient, which is always better in the Stroock-Varopoulos inequality. We show here a generalized Stroock-Varopoulos inequality.
 \begin{prop}\label{Stroock_Varo}
    Assume $\delta=\inf\limits_{s>0}\dfrac{\psi(s)}{\gamma_\psi^+(s)}>0$ and let $u\in {W}^{J,\Psi}(\mathbb{R}^N)$ such that $G(u),\, A(u)\in {W}^{J,\Psi}(\mathbb{R}^N)$, where $A$ and $ G$ satisfy $ G'\ge\left|\Psi( A ')\right|$. Then
    \begin{equation}\label{S-V-plap}
\mathcal{E}(u;G(u))\ge\frac{\delta q}pE( A(u)).
    \end{equation}
 \end{prop}

\begin{proof} The proof follows from a calculus estimate.
    For any $d> c$ we have that
    \begin{equation*}
    \begin{array}{rl}
    \Psi\left(| A(d)- A(c)|\right)&\leq \displaystyle\Psi\left(\int_{c}^{d}| A'(s)|\, ds\right)
\leq \displaystyle\gamma_\Psi^+(d-c) \Psi \left(\frac{1}{d-c}\int_{c}^{d}| A'(s)|\, ds\right) \\
&   \leq \displaystyle\frac{\gamma_\Psi^+(d-c)}{d-c}\int_{c}^{d}\Psi\left(| A'(s)|\,\right)ds
     \leq \displaystyle\frac{p}{q}\gamma_\psi^+(d-c)|G(d)-G(c)|\\
&    \leq \displaystyle \frac p{\delta q}\psi(d-c)|G(d)-G(c)|.
    \end{array}
    \end{equation*}
The same inequality is obtained for $d\le c$. We now deduce \eqref{S-V-plap} by choosing $d=u(x)$, $c=u(y)$ and integrate with respect to $J(x-y)\,dxdy$.
    \end{proof}
\qed

For instance in the case of a  sum of powers, $\Psi(s)=\displaystyle\sum_{i=1}^Mk_is^{p_i}$, $p_1<p_2<\cdots <p_M$, we have $\gamma_\psi^+(s)=\displaystyle\max\{s^{p_1-1},\,s^{p_M-1}\}$ and $\delta=\min\{k_1p_1,\,k_Mp_M\}$.

All the above inequalities hold also, with different constants, for nonlinearities that behave like a power, i.e., when they satisfy~\eqref{psi-plap} instead of~\eqref{set-psi}.
In particular in that case the integral Kato inequality and the Stroock-Varopoulos inequality  coincide, but for the coefficient, both giving
\begin{equation}\label{SV-plap}
\mathcal{E}(u;|u|^{r-1}u)\ge cE( |u|^{\frac{r+p-1}p}).
\end{equation}

We also obtain some  calculus inequalities needed in proving uniqueness results in the last sections. We borrow ideas from~\cite{Lindqvist} and~\cite{Friedrichs} that deal with the exact power case.
\begin{lem}
  \label{lem-clarkson} Let $\psi$ be a nonnegative, nondecreasing, continuous odd function and let $\psi=\Psi'$.
  \begin{enumerate}
  \item[$i)$] If $\psi$ satisfies
  \begin{equation}
    \label{raizconvex}
    \frac{s\psi'(s)}{\psi(s)}\ge1\qquad\text{for every } s\ne0,
  \end{equation}
  then
  \begin{equation}
    \label{clarkson1}
    \left(\psi(a)-\psi(b)\right)(a-b)\ge4\Psi\left(\frac{a-b}2\right).
  \end{equation}
  \item[$ii)$] If $\psi$ is concave in $(0,\infty)$ then
  \begin{equation}
    \label{clarkson2}
    \left(\psi(a)-\psi(b)\right)(a-b)\ge\psi'(|a|+|b|)(a-b)^2.
  \end{equation}
  \item[$iii)$]
  If $\psi$ satisfies
  \begin{equation}
    \label{concavemas}
    c_1|s|^{p-2}\le\psi'(s)\le c_2|s|^{p-2}\qquad\text{for some } 1<p<2 \text{ and every } s\ne0,
  \end{equation}
  then
  \begin{equation}
    \label{clarkson3}
    \left(\psi(a)-\psi(b)\right)(a-b)\ge \frac{c(\Psi(a-b))^{\frac2p}}{(\Psi(a)+\Psi(b))^{\frac{2-p}p}}.
  \end{equation}

  \end{enumerate}
\end{lem}

\begin{proof}
$i)$ We begin by proving a Clarkson inequality. Condition \eqref{raizconvex} implies that the function $g(s)=\Psi(\sqrt{|s|})$ is convex. Therefore
$$
\begin{array}{l}
\displaystyle \Psi\left(\frac{a+b}2\right)+\Psi\left(\frac{a-b}2\right)= g\left(\left(\frac{a+b}2\right)^2\right)+g\left(\left(\frac{a-b}2\right)^2\right) \\ [4mm]
\qquad\qquad\displaystyle \le g\left(\left(\frac{a+b}2\right)^2+\left(\frac{a-b}2\right)^2\right)= g\left(\frac{a^2+b^2}2\right) \\ [4mm]
\qquad\qquad\displaystyle \le \frac12\left(g(a^2)+g(b^2)\right)= \frac12\left(\Psi(a)+\Psi(b)\right).
\end{array}
$$
Now the convexity of $\Psi$ implies
$$
\Psi(a)\ge\Psi(b)+\Psi'(b)(a-b),
$$
and also
$$
\Psi\left(\frac{a+b}2\right)\ge\Psi(b)+\frac12\Psi'(b)(a-b),
$$
so that
$$
\Psi(a)+\Psi(b)\ge 2\Psi\left(\frac{a+b}2\right)+2\Psi\left(\frac{a-b}2\right) \ge2\Psi(b)+\Psi'(b)(a-b)+2\Psi\left(\frac{a-b}2\right).
$$
This gives
$$
\Psi(a)\ge\Psi(b)+\Psi'(b)(a-b)+2\Psi\left(\frac{a-b}2\right),
$$
and reversing the roles of $a$ and $b$,
$$
\Psi(b)\ge\Psi(a)+\Psi'(a)(b-a)+2\Psi\left(\frac{a-b}2\right).
$$
Adding these two inequalities we get \eqref{clarkson1}.

$ii)$ Developing the function $\Psi$ around the point $s=a$ we get
$$
\begin{array}{l}
\displaystyle\Psi(b)=\Psi(a)+\Psi'(a)(b-a)+(b-a)^2\int_0^1(1-s)\Psi''(a+s(b-a))\,ds \\
\qquad\qquad\displaystyle\ge\Psi(a)+\Psi'(a)(b-a)+(b-a)^2\Psi''(a+b)\int_0^1(1-s)\,ds.
\end{array}
$$
We have used that $|a+s(b-a)|\le |a|+|b|$ and $\Psi''$ is nonincreasing in $(0,\infty)$. Observe that though $\Psi''$ is singular at zero, the integral is convergent.
We conclude as before.

$iii)$ As $ii)$, using \eqref{concavemas} in the last step.

%$iv)$ It is immediate from the previous item.
\qed\end{proof}

To end this section devoted to the preliminary properties of $E$ and $F$, we point out a result on symmetrization that says that the energy $E(u)$ decreases when we replace $u$ by its  symmetric rearrangement (the radially deacreasing function with the same distribution function as $u$).

\begin{theorem}
  \label{th-symmet} If $u\in{W}^{J,\Psi}(\mathbb{R}^N)$ and $u^*$ is its decreasing rearrangement, then
  \begin{equation*}\label{symm}
    E(u)\ge E(u^*).
  \end{equation*}
\end{theorem}
This property is well known for the norm in $W_0^{\sigma/2,p}(\Omega)$, $0<\sigma\le2$, $p>1$, see \cite{Almgren-Lieb}, and is proved in \cite{CdP} for general kernels when $p=2$. The same proof can be used to get the  result in our situation, so we omit the details.
%%%%%%%%%%%%%%%%%%%%%%%%%%%%%%%%%%%%%%%%%%%%%%%%
\section{Sobolev inclusions}\label{sect-inclusions}
\setcounter{equation}{0}

In this section we consider a nonlinearity $\Psi\in\Gamma_{p,q}$,
$p\ge q>\max\{q_*,\,1\}$ fixed. As to the kernel $J$, besides
condition \eqref{kernel0} we also consider, for some results, the
singularity condition at the origin
\begin{equation}
  \label{alpha}
  J(z)\ge c|z|^{-N-\alpha}\quad \text{ for } 0<|z|<1,\quad\alpha>0.
\end{equation}Clearly it must be $\alpha\le q_*$.
In fact in the fractional $p$--Laplacian case it is $\alpha=\sigma p/2$. Other kernels could also be considered, for instance $J(z)=|z|^{-N-\mu}\left|\log (|z|/2)\right|^{\beta}$, for $0<|z|<1$, with $\mu\ge0$, (and $\beta\ge-1$ if $\mu=0$). In that case it is $q_*=\mu$. If $\mu>0$ then $J$ satisfies \eqref{alpha} with $\alpha=\mu$ if $\beta\ge0$, but if $\beta<0$ it satisfies \eqref{alpha} only with $0<\alpha<\mu$. A more intricate example can be constructed by the following piecewise definition of $J$,
$$
J(z)=\left\{\begin{array}{llc}
  |z|^{-N}&\text{ if } &2^{-2k-1}<|z|\le 2^{-2k}, \\ [2mm]
  |z|^{-N-\mu}&\text{ if } &2^{-2k}<|z|\le 2^{-2k+1},
\end{array}\right.
$$
$k\ge1$, $\mu>0$. Here we have $q_*=\mu$ while condition \eqref{alpha} does not hold for any $\alpha>0$.

Assume now that $u$ has  support contained in $\overline\Omega$. Then
\begin{equation*}\label{otraE}
\begin{array}{rl}
E(u)&\displaystyle
=\frac12\iint_{\mathbb{R}^{2N}}\Psi(u(x)-u(y))J(x-y)\,dxdy
\\ [4mm]
&\displaystyle
=\frac12\int_{\Omega}\int_{\Omega}\Psi(u(x)-u(y))J(x-y)\,dxdy+
\int_{\Omega}\int_{\Omega^c}\Psi(u(x))J(x-y)\,dxdy
\\ [4mm]
&\displaystyle\ge\int_{\Omega}\Psi(u(x))\Lambda(\Omega;x)\,dx,
\end{array}
\end{equation*}
where
\begin{equation*}\label{Lambda}
\Lambda(\Omega;x)=\int_{\Omega^c}J(x-y)\,dy.
\end{equation*}
If
$$
\mu=\min\{J(z)\,:\, |z|\le R\}>0,\qquad R>\delta=\sup_{x\in \Omega}dist(x,\Omega^c),
$$
then
$$
\Lambda(\Omega;x)\ge \mu|\{\delta<|z|<R\}|=A>0\qquad\text{ for every }x\in \Omega.
$$
This gives the Poincar\'e inequality
\begin{equation}\label{poincare1}
E(u)\ge AF(u),
\end{equation}
and the inclusion ${W}_0^{J,\Psi}(\Omega)\subset L^\Psi(\Omega)$.
We remark that in the case of integrable kernel $J$ we immediately would get $E(u)\le c\|J\|_1\,F(u)$, and thus ${W}_0^{J,\Psi}(\Omega)\equiv L^\Psi(\Omega)$.

In order to obtain better energy estimates in the case $q_*>0$,  which would yield better space embeddings, we need a better estimate of the function $\Lambda(\Omega;\cdot)$ in terms of the kernel $J$.
The following result is essentially contained in \cite[Lemma A.1]{Savin-Valdinoci2}.
 \begin{prop}\label{Lambda3}
 $$
\Lambda(\Omega;x)\ge P\left(\left(\frac{|\Omega|}{\omega_N}\right)^{1/N}\right)\qquad\text{ for every }x\in \Omega,
$$
where $P(s)=\int_{|z|>s}J(z)\,dz$. In particular, if $J$ satisfies \eqref{alpha} then
\begin{equation*}
  \label{Lambda2}
 \Lambda(\Omega';x)\ge c(\Omega)|\Omega'|^{-\alpha/N}\qquad\text{ for every } \Omega'\subset\Omega.
\end{equation*}
 \end{prop}
This estimate allows us to prove, assuming condition~\eqref{alpha}, the  Sobolev embedding ${W}_0^{J,\Psi}(\Omega)\subset L^{\Psi^r}(\Omega)$ for every $1\le r\le r^*\equiv \frac N{N-\alpha}$, if $\alpha<N$, for every $1\le r<\infty$ if $\alpha\ge N$. The proof uses ideas of \cite{Savin-Valdinoci} and \cite{hitch}. If $\alpha\ge N$  we obtain the result substituting $\alpha$ by any number below $N$ and close to $N$, since~\eqref{alpha} still holds for that exponent.
\begin{theorem}\label{th-Sobolev}
   Assume $J$ satisfies condition \eqref{alpha} with $0<\alpha<N$. Then there exists a positive constant $C=C(N,p,q,\alpha, \Omega)$ such that, for any  function $u\in{W}_0^{J,\Psi}(\Omega)$ we have $u\in L^{\Psi^r}(\Omega)$ for every $1\le r\le r^*\equiv \frac N{N-\alpha}$ and
    \begin{equation}\label{Sobolev}
    \|\Psi(u)\|_{r}\leq C E(u).
    \end{equation}
\end{theorem}

\begin{proof}
We prove the inequality for $r=r^*$, and then the result for $r<r^*$ follows by H\"older inequality.     We can assume, without loss of generality, that $u$ is radially deacreasing and $\Omega=B_{R^*}$, since substituting $u$ by its symmetric decreasing rearrangement $u^*$, we have by Theorem~\ref{th-symmet},

    \begin{equation*}
        \|\Psi(u)\|_{r}=\|\Psi(u^*)\|_{r}\leq  C E(u^*) \leq CE(u).
    \end{equation*}
We may also consider the case of $u$ bounded, since if not, taking the sequence $u_T=\min\{u,T\}$, and thanks to the Dominated Convergence Theorem, we would get the result in the limit $T\to\infty$.
We now define
    \begin{equation*}\label{R_k}
    \begin{array}{ll}
    A_k:=\{x\in\mathbb{R}^N\,:\,u(x)>2^k\},&\quad a_k=|A_k|, \\ [3mm]
    D_k:=A_k\setminus A_{k+1},&\quad d_k=|D_k|.
    \end{array}
    \end{equation*}
We have $A_k=B_{R_k}$, with  $R_{k+1}\leq R_k\leq R^*$. Also $a_k=d_k=0$ for all large $k$, say for $k> M$.
Now we compute,
\begin{equation*}
    \|\Psi(u)\|_{r}= \left(\sum_{k=-\infty}^M\int_{D_k}\Psi^r(u(x))\,dx\right)^{1/r}\leq \sum_{k=-\infty}^M\Psi(2^{k+1})d_k^{1/r}\leq c\sum_{k=-\infty}^M\Psi(2^{k})a_k^{1/r},
\end{equation*}
since $r>1$. On the other hand, if $x\in D_i$ and $y\in D_j$, with $j\le i-2$, then $$|u(x)-u(y)|\ge 2^i-2^{j+1}\ge2^{i-1}.$$
Thus
$$
\begin{array}{l}
\displaystyle\sum_{i=-\infty}^M\sum_{j=-\infty}^{i-2}\int_{D_i}\int_{D_j}\Psi(u(x)-u(y))J(x-y)\,dydx\\ [4mm]
\qquad\qquad\displaystyle
\ge
\sum_{i=-\infty}^M\Psi(2^{i-1})\int_{D_i}\sum_{j\le i-2}\int_{D_j}J(x-y)\,dydx \\ [4mm]
\qquad\qquad\displaystyle\ge
\sum_{i=-\infty}^M\Psi(2^{i-1})\int_{D_i}\int_{A^c_{i-1}}J(x-y)\,dydx \ge
c\sum_{i=-\infty}^M\Psi(2^{i})a_{i-1}^{-\alpha/N}d_i\\ [4mm]
\qquad\qquad\displaystyle=
c\sum_{i=-\infty}^M\Psi(2^{i})a_{i-1}^{-\alpha/N}\left(a_i-
\sum_{k=i+1}^{M}d_k\right)=c(A-B).
\end{array}
$$
The second term can be estimated as
$$
\begin{array}{rl}
B&=\displaystyle
\sum_{i=-\infty}^M\sum_{k=i+1}^{M}\Psi(2^{i})a_{i-1}^{-\alpha/N}d_k  =\sum_{k=-\infty}^{M}\sum_{i=-\infty}^{k-1}\Psi(2^{i})a_{i-1}^{-\alpha/N}d_k\\ [4mm]
&\displaystyle\le
\sum_{k=-\infty}^{M}\sum_{i=-\infty}^{k-1}\Psi(2^{i})a_{k-1}^{-\alpha/N}d_k\le
\sum_{k=-\infty}^{M}\Psi(2^{k})a_{k-1}^{-\alpha/N}d_k\sum_{i=-\infty}^{k-1}\gamma^+(2^{i-k})\\ [4mm]
&\displaystyle=
\sum_{k=-\infty}^{M}\Psi(2^{k})a_{k-1}^{-\alpha/N}d_k\sum_{m=1}^\infty\gamma^+(2^{-m})
=c\sum_{k=-\infty}^{M}\Psi(2^{k})a_{k-1}^{-\alpha/N}d_k=c(A-B).
\end{array}
$$
We deduce the estimate
$$
\begin{array}{rl}
E(u)&\displaystyle
\ge\sum_{i=-\infty}^M\sum_{j=-\infty}^{i-2}\int_{D_i}\int_{D_j}\Psi(u(x)-u(y))J(x-y)\,dydx \\ [4mm]
&\displaystyle\ge CA=C\sum_{i=-\infty}^M\Psi(2^{i})a_{i-1}^{-\alpha/N}a_i.
\end{array}$$
We conclude, using \cite[Lemma 5]{Savin-Valdinoci}, since $\frac 1r=1-\frac\alpha N$,
 $$
E(u)\ge C\sum_{i=-\infty}^M\Psi(2^{i})a_{i}^{1-\alpha/N}\ge C\|\Psi(u)\|_{r}.
$$
\qed
\end{proof}

We also prove that the above embedding is compact provided $r<r^*$. To this end we first show the compactness of the inclusion for $r=1$ and then interpolate with the continuity  for $r=r^*$. It is important to remark that the inclusion ${W}^{J,\Psi}_0(\Omega)\hookrightarrow L^{\Psi}(\Omega)$ is compact even when $q_*=0$, which implies $r^*=1$, provided the following conditions on the kernel at the origin hold
\begin{equation}
\label{singular}
\lim_{|z|\to0^+}|z|^{N}J(z)=\infty,
\end{equation}
\begin{equation}
\label{regular-v}
J(z_1)\ge cJ(z_2)\qquad \text{for every } 0<|z_1|\le |z_2|\le1, \text{ and some } c>0.
\end{equation}
This implies some kind of minimal singularity and some monotonicity near the origin. In particular this allows to consider for instance a kernel of the form $J(z)=|z|^{-N}\left|\log|z|\right|^{\beta}$, $\beta>0$, for $|z|\sim0$. See~\cite{CdP} for the case $\Psi(s)=|s|^2$.

\begin{theorem}\label{th-compact}
    Assume $J$ satisfies  \eqref{singular} and \eqref{regular-v}. Then the embedding ${W}^{J,\Psi}_0(\Omega)\hookrightarrow L^{\Psi}(\Omega)$ is compact.
\end{theorem}
\begin{proof}The idea of the proof goes back to the Riesz-Fr\'echet-Kolmogorov work. We follow the adaptation to the fractional Laplacian framework performed in \cite{hitch}.

Let $\mathcal{A}\subset{{W}^{J,\Psi}_0}(\Omega)$ be a bounded set. We show that $\mathcal{A}$ is totally bounded in $L^{\Psi}(\Omega),$ i.e., for any $\epsilon \in (0,1)$ there exist $\beta_1,...,\beta_M \in L^{\Psi}(B_1)$ such that for any $u\in \mathcal{A}$ there exists $j\in\{1,...,M\}$ such that
\begin{equation}\label{H131}
F(u-\beta_j)\leq \epsilon.
\end{equation}
We take a collection of disjoints cubes $Q_1,...Q_{M'}$ of side $\rho<1$ such that $\Omega=\bigcup_{j=1}^{M'} Q_j$.
For any $x\in \Omega$ we define $j(x)$ as the unique integer in $\{1,...,M'\}$ for which $x\in Q_{j(x)}$.
Also, for any $u \in \mathcal{A},$ let
$$
Q(u)(x):=\frac{1}{|Q_{j(x)}|}\int_{Q_{j(x)}}u(y)\,dy.
$$
Notice that
$$
Q(u+v)=Q(u)+Q(v) \;\mbox{for any}\; u,v\in \mathcal{A},
$$
and that $Q(u)$ is constant, say equal to $q_j(u)$, in any $Q_j$, for $j\in\{1,...,M'\}$.
Therefore, we can define
$$
S(u):=\rho^{N}\left(\Psi(q_1(u)),...,\Psi(q_{M'}(u))\right)\in \R^{M'},
$$
and consider the spatial $1$-norm in $\R^{M'}$ as
$$
\|v\|_1:=\displaystyle\sum_{j=1}^{M'}|y_j|, \qquad\mbox{ for any } v=(y_1,\dots,y_{M'})\in \R^{M'}.
$$
We observe that
\begin{equation}\label{cota_norm_P}
\begin{array}{rl}
\displaystyle F(Q(u))&\displaystyle=
\sum_{j=1}^{M'}\int_{Q_j}\Psi\left(Q(u)(x)\right)\,dx\displaystyle\leq \rho^N\sum_{j=1}^{M'}\Psi\left(q_j(u)\right)\\
&\leq\displaystyle \rho^N\sum_{j=1}^{M'}\Psi\left(q_j(u)\right)
=\|S(u)\|_1,
\end{array}
\end{equation}
and also, by Jensen inequality and \eqref{poincare1},
\begin{equation}\label{S(u)}
\begin{array}{rl}
\|S(u)\|_1&\displaystyle=\sum_{j=1}^{M'}\rho^N\left|\Psi(q_j(u))\right|=\rho^N\sum_{j=1}^{M'}\Psi\left(\frac{1}{\rho^N}\int_{Q_j}u(y)\,dy\right)\\
&\displaystyle\leq \sum_{j=1}^{M'}\int_{Q_j}\Psi(u(y))\,dy=\int_{\Omega}\Psi(u(y))\,dy
\leq c.
\end{array}
\end{equation}
In the same way,
$$
F(Q(u)-a)=F(Q(u)-Q(a))\le \|S(u)-S(a))\|_1
$$
for every constant $a$.
In particular from \eqref{S(u)} we obtain that the set $S(\mathcal{A})$ is bounded in $\R^{M'}$ and so, since it is finite dimensional, it is totally bounded. Therefore, there exist $b_1,...,b_K\in\R^{M'}$ such that
\begin{equation}\label{H134}
S(\mathcal{A})\subset\bigcup_{i=1}^{K}B_{\eta}(b_i),
\end{equation}
where $B_{\eta}(b_i)$ are the $1$--balls of radius $\eta$ centered at $b_i$.
For any $i\in\{1,...,K\}$, we write the coordinates of $b_i$ as $b_i=(b_{i,1},...,b_{i,M'})\in \R^{M'}$. For any $x\in \Omega$ we set
$$
\beta_i(x)=\Psi^{-1}(\rho^{-N}b_{i,j(x)}),
$$
where $j(x)$ is as above.
Notice that $\beta_i$ is constant on $Q_j$, i.e. if $x\in Q_j$ then
\begin{equation}\label{constant}
Q(\beta_i)(x)=\Psi^{-1}(\rho^{-N}b_{i,j(x)})=\Psi^{-1}(\rho^{-N}b_{i,j})=\beta_{i}(x)
\end{equation}
and so $q_j(\beta_i)=\Psi^{-1}(\rho^{-N}b_{i,j})$; thus $S(\beta_i)=b_i$.
Furthermore,  again by Jensen inequality,

\begin{equation*}
\begin{array}{rl}
\displaystyle F(u-Q(u))&
= \displaystyle\sum_{j=1}^{M'}\int_{Q_j}\Psi\left(u(x)-Q(u)(x)\right)\,dx\\
&\displaystyle=\sum_{j=1}^{M'}\int_{Q_j}\Psi\left(\frac{1}{\rho^N}\int_{Q_j}\left(u(x)-u(y)\right)\,dy\right)\,dx \\
&\displaystyle\leq \frac{1}{\rho^N}\displaystyle\sum_{j=1}^{M'}\int_{Q_j}\int_{Q_j}\Psi\left(u(x)-u(y)\right)\,dy\,dx\\
&\displaystyle\leq\frac{1}{\ell(\rho)}\sum_{j=1}^{M'}\int_{Q_j}\int_{Q_j}\Psi\left(u(x)-u(y)\right)J(x-y)\,dy\,dx \leq\frac{c}{\ell(\rho)},
\end{array}
\end{equation*}
where $\ell(\rho)=\rho^NJ(\rho)$, using \eqref{regular-v}.
Consequently, for any $j\in\{1,...,K\}$, recalling \eqref{cota_norm_P} and \eqref{constant}
\begin{equation*}\label{H137}
\begin{array}{rl}
F(u-\beta_j)&\displaystyle\leq F(u-Q(u))+F(Q(u)-Q(\beta_j))+
F(Q(\beta_j)-\beta_j)\\ [3mm]
&\displaystyle\leq c\left(\frac{1}{\ell(\rho)}+
\|S(u)-S(\beta_j))\|_1\right).
\end{array}
\end{equation*}
Now recalling \eqref{H134}   we take $j\in \{1,...,K\}$ such that $S(u)\in B_{\eta}(b_j)$, that is
$$
\|S(u)-S(\beta_j))\|_1=\|S(u)-b_j\|_1<\eta.
$$
We conclude by choosing $\rho$ and $\eta$ small, thanks to \eqref{singular}, so as to have $c\left(\frac{1}{\ell(\rho)}+\eta\right)<\epsilon.$
\qed\end{proof}

As a corollary we obtain the full compactness result in the fractional case.
\begin{theorem}\label{th-compact2}
    Assume $J$ satisfies  \eqref{alpha} and \eqref{regular-v}. Then the embedding ${W}^{J,\Psi}_0(\Omega)\hookrightarrow L^{\Psi^r}(\Omega)$ is compact for every $1\le r<r^*$ if $\alpha<N$, for every $1\le r<\infty$ if $\alpha\ge N$.
\end{theorem}
\begin{proof}
As before if $\alpha\ge N$  we obtain the result substituting $\alpha$ by any number below $N$. By  classical interpolation
$$
\|\Psi(u)\|_r\le \|\Psi(u)\|_1^\lambda\|\Psi(u)\|_{r^*}^{1-\lambda}\le cF(u)^\lambda,
$$
where $\frac1r=\lambda+\frac{1-\lambda}{r^*}$. Therefore we can obtain, instead of \eqref{H131}, the estimate
$$
\int_\Omega\Psi^r(u(x)-\beta_j)\,dx\le c\epsilon^{\lambda r},
$$
and we are done.
\qed\end{proof}

%%%%%%%%%%%%%%%%%%%%%%%%%%%%%%%%%%%%%%%%%%%%%%%%
\section{The  problem with reaction $f=f(x)$}\label{sect-linear}
\setcounter{equation}{0}

We start with this section the study of some elliptic type problems associated to our nonlinear nonlocal operator $\mathcal{L}$.
%We assume throughout these sections that the constitutive (odd) function $\psi$ satisfies condition \eqref{psi-plap}.

Here we consider the problem
\begin{equation}\label{linearproblem}
\begin{cases}
\mathcal{L}u=f(x),& \mbox{in }\Omega,\\ u=0,& \mbox{in }\Omega^c.
\end{cases}
\end{equation}
Given any $f\in \left({W}^{J,\Psi}_0(\Omega)\right)'$, the dual space, we say that $u\in {W}^{J,\Psi}_0(\Omega)$ is a weak solution to \eqref{linearproblem} if \eqref{weaksol} holds.

By  Poincar\'e inequality \eqref{poincare1} we have that $f\in \left({W}^{J,\Psi}_0(\Omega)\right)'$ for instance provided $f\in L^\Phi(\Omega)$, where $\Phi$ is the complementary function of $\Psi$.

We next show that problem \eqref{linearproblem} has a weak solution. We do not know if this solution is a strong solution, that is if $\mathcal{L}u$ is defined pointwise and the equality in~\eqref{linearproblem} holds almost everywhere. On the other hand, we are able to show uniqueness assuming some extra conditions on the function $\Psi$. In the exact power case $\Psi(s)=|s|^{p}$ these extra conditions cover the full range $p>1$.

\begin{theorem}\label{th-existence-linear}
For    any $f\in \left({W}^{J,\Psi}_0(\Omega)\right)'$ there exists a  solution $u\in {W}^{J,\Psi}_0(\Omega)$ to problem \eqref{linearproblem}. If $\psi$ satisfies either condition~\eqref{raizconvex} or~\eqref{concavemas} then the solution is unique.
\end{theorem}
\begin{proof}
Existence follows by minimizing in ${W}^{J,\Psi}_0(\Omega)$ the functional
\begin{equation*}
  \label{functional-f}
  I(v)=E(v)-\int_\Omega fv.
\end{equation*}
Clearly it is well defined, lower semicontinuous and Fr\'echet differentiable with
$$
\langle I'(v),\varphi\rangle=\mathcal{E}(v;\varphi)-\int_\Omega f\varphi
$$
for every $v,\,\varphi\in {W}^{J,\Psi}_0(\Omega)$. To see that it is coercive we first observe that $\|v\|_{W^{J,\Psi}}\to\infty$ implies $E(v)\to\infty$. Actually, by~\eqref{equiv-norms}, \eqref{equiv-norms2} and Poincar\'e inequality,
$$
\begin{array}{rl}
\|v\|_{W^{J,\Psi}}&\displaystyle= \|v\|_{{L}^{\Psi}}+[v]_{W^{J,\Psi}}\le
c\left((\gamma^-_\Psi)^{-1}(F(v))+(\gamma^-_\Psi)^{-1}(E(v))\right) \\ [3mm]
&\displaystyle\le c\max\left\{\left(E(v)\right)^{1/p},\, \left(E(v)\right)^{1/q}\right\}.
\end{array}$$
Now use H\"older inequality in Orlicz spaces,
$$
\left|\int_\Omega fv\right|\le c\|v\|_{{W}^{J,\Psi}_0},\qquad c=\sup_{\|w\|_{{W}^{J,\Psi}_0}=1}\left|\int_\Omega  fw\right|.
$$
The last quantity is known as the Orlicz norm of $f$ in $\left({W}^{J,\Psi}_0(\Omega)\right)'$, an is equivalent to the Luxemburg norm, see~\cite{Rao-Ren}. We thus get
$$
I(v)\ge E(v)-c\left(E(v)\right)^{1/q}\to\infty
$$
as $\|v\|_{W^{J,\Psi}}\to\infty$.
Therefore there exists a minimum of $I$, attained by compactness for some function $u\in {W}^{J,\Psi}_0(\Omega)$, which is a weak solution to our problem.

We now show uniqueness. Suppose by contradiction that there exist two functions $u_1,\,u_2\in  {W}^{J,\Psi}_0(\Omega)$ such that
\begin{equation}\label{iguales}
\mathcal{E}(u_1;\varphi)=\mathcal{E}(u_2;\varphi)\qquad \forall\;\varphi\in  {W}^{J,\Psi}_0(\Omega).
\end{equation}
Assume first that \eqref{raizconvex} holds.
We have, denoting $a=u_1(x)-u_1(y)$, $b=u_2(x)-u_2(y)$, and using \eqref{clarkson1},
$$
\begin{array}{rl}
E(u_1-u_2)&\displaystyle=\frac12\iint_{\mathbb{R}^{2N}}\Psi(a-b)J(x-y)\,dxdy \\ [4mm]
&\displaystyle\le c\iint_{\mathbb{R}^{2N}}\left(\psi(a)-\psi(b)\right)(a-b)J(x-y)\,dxdy \\ [4mm]
&\displaystyle= c\left(\mathcal{E}(u_1;u_1-u_2)-\mathcal{E}(u_2;u_1-u_2)\right)=0
\end{array}
$$
by \eqref{iguales}.
This implies $u_1\equiv u_2$.

Assume now condition \eqref{concavemas}. We calculate, using H\"older inequality and \eqref{clarkson3},
$$
\begin{array}{l}
E(u_1-u_2)\displaystyle=\frac12\iint_{\mathbb{R}^{2N}}\Psi(a-b)J(x-y)\,dxdy \\ [4mm]
\displaystyle\le c\left(\iint_{\mathbb{R}^{2N}}\frac{\left(\Psi(a-b)\right)^{2/p}}{(\Psi(a)+\Psi(b))^{\frac{2-p}p}}J(x-y)
\,dxdy\right)^{\frac p2}\hspace{-1mm} \left(\iint_{\mathbb{R}^{2N}}\left(\Psi(a)+\Psi(b)\right)J(x-y)\,dxdy\right)^{1-\frac p2} \\ [4mm]
\displaystyle\le c\left(\iint_{\mathbb{R}^{2N}}\left(\psi(a)-\psi(b)\right)(a-b)J(x-y)\,dxdy\right)^{\frac p2}\hspace{-1mm}
\left(\iint_{\mathbb{R}^{2N}}\left(\Psi(a)+\Psi(b)\right)J(x-y)\,dxdy\right)^{1-\frac p2}\\ [4mm]
\displaystyle= c\left(\mathcal{E}(u_1;u_1-u_2)-\mathcal{E}(u_2;u_1-u_2)\right)^{\frac p2}
\left(E(u_1)+E(u_2)\right)^{1-\frac p2}=0.
\end{array}
$$
\qed\end{proof}

A maximum principle is easy to obtain.

 \begin{prop}\label{maximum}
 	If $u\in\mathcal{H}^{J,\Psi}(\R^N)$ then
 $$
 \left.
 \begin{array}{l}\mathcal{E}(u, \varphi)\ge0 \quad \forall\;\varphi \in \mathcal{H}^{J,\Psi}(\R^N), \;\varphi\ge0 \\
 u\ge0 \;\text{ in } \Omega^c
 \end{array}
 \right\}
 \;\Rightarrow\;u\geq0\text{ in } \Omega.
 $$
 \end{prop}
\noindent{\it Proof.} Since $u^-\ge0$ and $u^-\in\mathcal{H}^{J,\Psi}(\R^N)$,  we have, by \eqref{kato-simple},
$$
0\ge- \mathcal{E}(u^-, u^-)\ge \mathcal{E}(u, u^-)\ge 0.
$$
Hence $u^-\equiv0$.
\qed

We now study the integrability properties of the solution in terms of the integrability of the datum in the power-like case~\eqref{psi-plap}. In the exact power case of the fractional $p$--Laplacian these integrability properties have been obtained in~\cite{BarriosPeralVita}. Our proofs in the more general case treated in this paper differ   from theirs in  that we are using Stroock-Varopoulos inequality instead of Kato inequality, and that we allow for the limit case $q_*=0$, which does not make sense in the fractional $p$--Laplacian. All the proofs are based on the well known Moser iteration technique for the standard Laplacian case, see for example the book~\cite{Gilbarg-Trudinger}.

The first result uses no singularity condition on the kernel $J$, besides being nonintegrable.
 \begin{theorem}\label{th-Lp1}
Assume condition \eqref{psi-plap}.  If $u$ is a weak solution to problem \eqref{linearproblem} with $f\in L^m(\Omega)$  then $u\in L^{m(p-1)}(\Omega)$.
\end{theorem}
Of course this result is not trivial only if $m>\frac p{p-1}$, since $u$ being a weak solution it belongs to ${W}^{J,\Psi}_0(\Omega)\subset L^p(\Omega)$.

\begin{proof}
Without loss of generality we may assume $u\ge0$, and
this simplifies notation; the general case is obtained in a similar way. We define for $\beta\geq 1$ and $K>0$  the function
$$
    H(s)=\left\{
    \begin{array}
      {l@{\qquad}l}
      s^\beta,&s\le K,\\
      \mbox{linear},&s>K.
    \end{array}
    \right.
$$
We choose as test function $ \varphi=G(u)=\int_0^u \Psi(H'(s))\,ds.$ It is easy to check that $\varphi\in {W}^{J,\Psi}_0(\Omega)$. In fact
$$
E(\varphi)\le\gamma^+_\Psi(\Psi(\beta K^{\beta-1}))E(u)<\infty.
$$
We obtain
on one hand, using the Stroock-Varopoulos inequality \eqref{S-V-plap} and the Poincar\'e inequality \eqref{poincare1},
 \begin{equation}
   \label{GT01}
     \mathcal{E}(u;G(u))\ge cE(H(u))\ge cF(H(u)),
  \end{equation}
and on the other hand, using H\"older inequality,
 \begin{equation}
   \label{GT02}
   \int_\Omega fG(u)\leq \|f\|_m\|G(u)\|_{m'}.
 \end{equation}
Letting $K\to \infty$ in the definition
of $H$, the inequalities~\eqref{GT01} and~\eqref{GT02} give
\begin{equation}\label{integrable1}
  \|u\|_{p\beta}^{p\beta}\leq c\|f\|_m\|u\|_{((\beta-1)p+1)m'}^{(\beta-1)p+1}\,.
\end{equation}
Choosing now $\beta=\frac{m(p-1)}p$, we get
\begin{equation*}\label{estim-Lp}
\|u\|_{m(p-1)}\le c\|f\|^{\frac1{p-1}}_m\,.
\end{equation*}

\qed
\end{proof}

The same proof allows to gain more integrability  when condition \eqref{alpha} holds.
\begin{theorem}\label{th-Lp2}
  Assume conditions \eqref{psi-plap} and \eqref{alpha} and let $u$ be a weak solution to problem \eqref{linearproblem}, where $f\in L^m(\Omega)$, $m<N/\alpha$. Then $u\in L^{\frac{m(p-1)N}{N-m\alpha}}(\Omega)$.
\end{theorem}
Again this result is not trivial only if $m>\frac{Np}{Np-N+\alpha}$, since then $\frac{m(p-1)N}{N-m\alpha}>\frac{Np}{N-\alpha}$.

\begin{proof}
In the previous proof, using Sobolev inequality \eqref{Sobolev} instead of Poincar\'e inequality, we obtain in \eqref{integrable1}
\begin{equation*}\label{integrable2}
  \|u\|_{p\beta r^*}^{p\beta}\leq \beta\|f\|_m\|u\|_{((\beta-1)p+1)m'}^{(\beta-1)p+1},
\end{equation*}
$r^*=\frac N{N-\alpha}$. Choosing now $\beta=\frac{m'(p-1)}{p(r^*-m')}$, we get
\begin{equation*}\label{estim-Lp2}
\|u\|_{\frac{m(p-1)N}{N-m\alpha}}\le c\|f\|^{\frac1{p-1}}_m.
\end{equation*}
  \qed
\end{proof}

Even more, assuming a better integrability condition on $f$ we get that the solution is bounded.  This is a well known result for the standard Laplacian or the fractional Laplacian.
    \begin{theorem}\label{Linf}
  Assume conditions \eqref{psi-plap} and \eqref{alpha}. If $u$ is a weak solution to problem \eqref{linearproblem}, where $f\in L^m(\Omega)$ with $m>N/\alpha$, then $u\in L^\infty(\Omega)$.
\end{theorem}
\begin{proof}
We change here slightly the test function used in the previous two proofs.
We define for $\beta\geq 1$ and $K\geq k$ ($k$ to be chosen later) a ${\cal
C}^1([k,\infty))$ function $H$, as follows:
$$
    H(s)=\left\{
    \begin{array}
      {l@{\qquad}l}
      s^\beta-k^\beta,&s\in[k,K],\\
      \mbox{linear},&s>K.
    \end{array}
    \right.
$$
Let us also define $v=u+k$, and choose as test function $  \varphi=G(v)=\int_k^v \Psi(H'(s))\,ds.$
We obtain
on one hand, using the Stroock-Varopoulos inequality \eqref{S-V-plap} and the Sobolev inequality \eqref{Sobolev},
 \begin{equation}
   \label{GT1}
     \mathcal{E}(u;G(v))\ge cE(H(v))\ge c\|\Psi(H(v))\|_{r^*}\,,
  \end{equation}
and on the other hand, using H\"older inequality,
 \begin{equation}
   \label{GT2}
   \int_\Omega fG(v)\leq \int_\Omega fv\Psi(H'(v))\leq \frac{1}{k^{p-1}}\int_\Omega fv^p\Psi(H'(v))\leq \frac{c}{k^{p-1}}\|f\|_m\|vH'(v)\|^p_{pm'}\,,
 \end{equation}
 since $v\ge k$.
Inequality~\eqref{GT1} together with~\eqref{GT2}, and the properties of $\Psi$, lead to
\begin{equation}\label{eq:ineq}
  \|H(v)\|_{r^*p}\leq \left(\frac{c\|f\|_{m}}{k^{p-1}}\right)^{1/p}
  \|vH'(v)\|_{pm'}\,.
\end{equation}
We choose $k=(c\|f\|_m)^{\frac1{p-1}}$ and let $K\to \infty$ in the definition
of $H$, so that the inequality~\eqref{eq:ineq} becomes
$$
  \|u\|_{r^*p\beta}\leq \beta\|u\|_{pm'\beta}.
$$
Hence for all $\beta\geq 1$ the inclusion $u\in L^{pm'\beta}(\Omega)$ implies the stronger inclusion $u\in
L^{r^*p\beta}(\Omega)$, since $r^*=\frac{N}{N-\alpha}>m'=\frac{m}{m-1}$ provided $m>\frac N\alpha$. Observe that $u$ being a weak solution it belongs to ${W}^{J,\Psi}_0(\Omega)$, and thus $u\in
L^{\frac{Np}{N-\alpha}}(\Omega)$. The result follows now iterating the estimate starting with
$\beta=\frac{N(m-1)}{(N-\alpha)m}>1$, see for example  \cite[Theorem 8.15]{Gilbarg-Trudinger} for the details in the standard Laplacian case.  This gives $u\in L^\infty(\Omega)$. In fact we get the estimate
$$
\|u\|_\infty\le c(E(u)^{\frac1p}+\|f\|^{\frac1{p-1}}_m).
$$

  \qed
\end{proof}

%%%%%%%%%%%%%%%%%%%%%%%%%%%%%%%%%%%%%%%%%%%%%%%%
\section{The  problem with reaction $f=f(u)$}\label{sect-nonlinear}
\setcounter{equation}{0}

We study in this section  the nonlinear elliptic type problem
 \begin{equation}\label{sublinear}
	\begin{cases}
		\mathcal{L}u=f(u),& \mbox{in }\Omega,\\ u\ge 0,\;u\not\equiv0,& \mbox{in }\Omega,\\ u=0,& \mbox{in }\Omega^c.
	\end{cases}
\end{equation}
We first show existence in the \emph{lower} case, i.e.,
when $f:\left[0,\infty\right) \to \R$ is a continuous function satisfying
\begin{equation}
\label{sub1}
\exists\;0<\mu<\frac{q-1}p\;:\qquad |f(t)|\le c_1+c_2\Psi^\mu(t),\qquad \liminf_{t\to0^+}\frac{f(t)}{\Psi^\mu(t)}\ge c_3>0.
\end{equation}
In the power-like case \eqref{psi-plap} with $f(t)=t^{m-1}$ this means $0<m<p$. See \cite{BrezisOswald} for the classical sublinear problem for $\mathcal{L}=-\Delta$ and~\cite{CdP} for general $\mathcal{L}$ with $q_*\ge0$, both in the case $\Psi(s)=|s|^2$.

\begin{theorem}\label{th-sublinear}
Under the assumption \eqref{sub1} problem \eqref{sublinear} has a solution $u\in  {W}^{J,\Psi}_0(\Omega)$. 	
\end{theorem}
\begin{proof}
We define the energy functional $I:{W}^{J,\Psi}_0(\Omega)\to \R $ defined by
$$
I(v)=E(v)-\int_{\Omega}G(v),
$$
where $G(u)=\int_{0}^{u}f(s)\,ds.$ This functional is easily seen to be  weakly lower semicontinuous, and is well defined since
\begin{equation}\label{welldefined}
\left|\int_{\Omega}G(v)\right|\leq  c_1|\Omega|+c_2|\Omega|^{1-\mu}(F(v))^{\mu}<\infty.
\end{equation}
On the other hand, this same estimate also gives coercivity since $
\mu<1$, and then
$$
I(v)\geq E(v)-c\left(E(v)\right)^{\mu}\to\infty\qquad\text{as } \|v\|_{W^{J,\Psi}}\to\infty.
$$

Let now $\{v_n\}\subset{W}^{J,\Psi}_0(\Omega)$ be a minimizing sequence for $I$,
$$
\liminf_{n\to \infty}I(v_n)=\nu=\inf_{u\in {W}^{J,\Psi}_0(\Omega)}I(u).
$$
This sequence is bounded in ${W}^{J,\Psi}_0(\Omega)$, and therefore we can assume that there is a subsequence, still denoted $\{v_n\}$, such that $v_n\rightharpoonup u$ in ${W}^{J,\Psi}_0(\Omega)$. Therefore $v_n\to u$ in $L^\Psi(\Omega)$. We thus deduce by~\eqref{welldefined}
$$
\int_{\Omega}G(v_n) \to \int_{\Omega}G(u),
$$
so that
$$
\nu\le I(u)\leq \liminf_{n\to \infty}\left(E(v_n)-\int_{\Omega}G(v_n)\right)=\liminf_{n\to \infty}I(v_n)=\nu.
$$
This shows that $I(u)=\nu$ and $u$ is a global minimum for $I$,  hence  a solution to \eqref{sublinear}. It is easy to see that we can replace $u$ by $|u|$ since $I(|u|)\le I(u)$. In order to show that $u$ is nontrivial let us check that $I(u)<0$. In fact, given any $v\in {W}^{J,\Psi}_0(\Omega)$ we have
$$
I(\varepsilon v)\le \gamma^+_\Psi(\varepsilon)E(v)-\varepsilon\left(\gamma^-_\Psi(\varepsilon)\right)^\mu\int_{\Omega}G(u)
\le \varepsilon^qE(v)-\varepsilon^{1+p\mu}\int_{\Omega}G(u)<0,
$$
for small $\varepsilon>0$, since $q>1+p\mu$. We deduce that $\nu<0$ and $u\not\equiv0$.
\end{proof}\qed

Unfortunately we are only able to prove uniqueness in the exact power case $\Psi(s)=|s|^p$. In fact uniqueness follows in that case using a standard argument by means of  a Picone inequality proved in \cite{Frank-Seiringer}, see~\cite{BrezisOswald} and~\cite{CdP}. Though a Picone inequality could be obtained also assuming that \eqref{psi-plap} is satisfied, it is not sharp enough to prove uniqueness. In the more general case of $\Psi\in \Gamma_{p,q}$ such Picone type inequality is not even known to hold.

We now assume condition \eqref{alpha} and consider  nonlinear functions $f$ in the intermediate range, that is above the power $p-1$ but  subcritical in the sense of Sobolev, see Theorem~\ref{th-Sobolev}. The precise conditions on $f$ are
\begin{equation}\label{rho}
\begin{array}{lll}
\exists\; \rho>p \;:\; &tf(t)\ge  \rho G(t) &\quad\forall\; t>0; \\ [3mm]
  \exists\;  1<r<r^*,\,t_0>0 \;:\;&tf(t)\le c\psi^r(t) &\quad\forall\;t>t_0; \\ [3mm]
\exists\; \lambda_0>0 \;:\; &f(\lambda t)\ge \lambda^\rho f(t) &\quad\forall\; t>0,\,\lambda>\lambda_0,
  \end{array}
\end{equation}
where $G'=f$.  When $f(t)=t^{m-1}$ these condition hold with $\rho=m$ provided $p<m<\frac{Nq}{N-\alpha}$.
\begin{theorem}\label{th-subcritical}
Assume $J$ satisfies~\eqref{alpha}, $\psi$ satisfies either~\eqref{raizconvex} or~\eqref{concavemas}, and $f$ is a nondecreasing function satisfying \eqref{rho}. Then  problem \eqref{sublinear} has a solution $u\in  {W}^{J,\Psi}_0(\Omega)$. 	
\end{theorem}
\begin{proof}
As before we consider the  functional
$$
I(v)=E(v)-\int_\Omega G(v),
$$
whose critical points are the solutions to our problem. This functional is well defined in ${W}^{J,\Psi}_0(\Omega)$ thanks to the Sobolev embedding and the second condition in~\eqref{rho}. We therefore apply the standard variational technique based on the Mountain Pass Theorem~\cite{Ambrosetti-Rabinowitz}. We only have to prove that the functional satisfies the Palais-Smale condition and has the appropriate geometry.

We  first prove that any Palais-Smale sequence has a convergent subsequence. Let $\{v_n\}$ be a sequence satisfying
$$
I(v_n)\to \nu,\qquad \langle I'(v_n),\varphi\rangle\to0\quad\forall\;\varphi\in\left({W}^{J,\Psi}_0(\Omega)\right)'.
$$
By the first condition in~\eqref{rho}, and using~\eqref{equiv-Ee}, we have
$$
\langle I'(v_n),v_n\rangle =\mathcal{E}(v_n;v_n)-\int_\Omega  v_nf(v_n)\le pE(v_n)-\rho\int_\Omega  G(v_n).
$$
On the other hand, for all large $n$,
$$
|\langle I'(v_n),v_n\rangle|\le \varepsilon_n\|v_n\|_{W^{J,\Psi}}.
$$
Therefore
$$
\begin{array}{rl}
\nu+1&\displaystyle\ge I(v_n)=I(v_n)-\frac1\rho\langle I'(v_n),v_n\rangle+\frac1\rho\langle I'(v_n),v_n\rangle\\ [3mm]
&\displaystyle\ge \left(1-\frac p\rho\right)E(v_n)-\frac{\varepsilon_n}\rho\|v_n\|_{W^{J,\Psi}} \\ [3mm]
&\displaystyle\ge\left(1-\frac p\rho\right)
\min\left\{\|v_n\|_{W^{J,\Psi}}^p,\,\|v_n\|_{W^{J,\Psi}}^q\right\}-
\frac{\varepsilon_n}\rho\|v_n\|_{W^{J,\Psi}}.
\end{array}
$$
This implies $\|v_n\|_{W^{J,\Psi}}\le k$ for every $n$, so that there exists a subsequence, still denoted $\{v_n\}$, converging weakly to some $u\in{W}^{J,\Psi}_0(\Omega)$, and by Theorem~\ref{th-compact2} it is $v_n\to v_\infty$ strongly in $L^{\Psi^r}(\Omega)$ for every $1\le r<r^*$. The second condition in~\eqref{rho} implies  $v_nf(v_n)\to v_\infty f(v_\infty)$ in $L^1(\Omega)$. Now write,
$$
\begin{array}{rl}
\mathcal{E}(v_n;v_n-v_\infty)-\mathcal{E}(v_\infty;v_n-v_\infty)&\displaystyle=\langle I'(v_n),v_n-v_\infty\rangle-\langle I'(v_\infty),v_n-v_\infty\rangle \\ [3mm]
&\displaystyle+\int_\Omega (f(v_n)-f(v_\infty))(v_n-v_\infty)\to0.
\end{array}
$$
Using inequalities~\eqref{clarkson1} or~\eqref{clarkson3}  as in the proof of uniqueness in Theorem~\ref{th-existence-linear}, we obtain
$$
E(v_n-v_\infty)\to0,
$$
that is $v_n\to v_\infty$ in ${W}^{J,\Psi}_0(\Omega)$, and Palais-Smale condition holds.

Let us now look at the behaviour of $I$ close to the origin and far from it. First $I(0)=0$. Also, given any $v\in{W}^{J,\Psi}_0(\Omega)$, we have by Poincar\'e inequality and the second condition in~\eqref{rho}
$$
I(v)=E(v)-\int_\Omega  G(v)\ge c_1\int_\Omega \psi(v)-c_2\int_\Omega \psi^r(v)\ge c_1F(v)-c_3F^r(v)>0
$$
for every $F(v)$ small. But $\|v\|_{W^{J,\Psi}}$ small implies $F(v)$ small.  We have obtained
$$
\exists\;\varepsilon>0\;:\; I(v)>I(0)\quad\forall\;v\in{W}^{J,\Psi}_0(\Omega),\;\|v\|_{{W}^{J,\Psi}}=\varepsilon.
$$
On the other hand, if $\lambda>0$ is large, using the third condition in~\eqref{rho}, we get
$$
I(\lambda v)\le \lambda^pE(v)-\lambda^\rho\int_\Omega  G( v)<0,
$$
since $p<\rho$. Thus
$$
\exists\;\overline v\in{W}^{J,\Psi}_0(\Omega),\; \|\overline v\|_{{W}^{J,\Psi}}>\varepsilon\;:\;I(\overline v)<I(0).
$$
This ends the proof by an application of the Mountain Pass Theorem. Actually, if we define
$$
\Theta=\{h\in C([0,1];{W}^{J,\Psi}_0(\Omega))\;:\;h(0)=0,\,h(1)=\overline v\},
$$
then
$$
\eta=\inf_{h\in\Theta}\max_{t\in[0,1]}I(h(t))
$$
is a critical value with $I(u)=\eta$ for some $u\in{W}^{J,\Psi}_0(\Omega)$, which is a solution to our problem.

\end{proof}\qed

The exponent $r^*$ in \eqref{rho} is sharp in the fractional $p$--Laplacian case. In fact, in the fractional Laplacian case $p=2$ this has been proved in~\cite{RosOtonSerra} by means of a Pohozaev identity when $\Omega$ is star-shaped. Their proof was adapted in \cite{CdP} for more general kernels again with $\Psi(s)=|s|^2$, obtaining an exponent which depends on the kernel and is presumed not to be optimal. The proof of this last result works verbatim for general powers $\Psi(s)=|s|^p$, but not for other functions, since homogeneity is crucial in the argument.

Let, for $\lambda>1$,
\begin{equation}
  \label{mu}
  \mu(\lambda)=\lambda^{-N}\sup_{\substack{z\in\mathbb{R}^N\\ z\ne0}}\frac{J(z/\lambda)}{J(z)},
\end{equation}
and assume $\mu(\lambda)<\infty$ for $\lambda$ close to 1.
\begin{theorem}
\label{teo-pohozaev}
If $u$ is a bounded   solution to problem \eqref{sublinear} with $\Psi(s)=|s|^p$ and $\Omega$ is star-shaped, then
\begin{equation*}
  \label{pohozaev}
  \int_\Omega uf(u)\le\frac{Np}{N-\delta}\int_\Omega G(u),
\end{equation*}
where $\delta=\mu'(1^+)$ and $G'=f$.
\end{theorem}

\begin{cor}\label{cor-p^*}
  Problem \eqref{sublinear} with $f(u)=u^{m-1}$, $\Psi(s)=|s|^p$ and $\Omega$ star-shaped has no bounded   solutions for any exponent $m>m_*=\frac{Np}{N-\delta}$.
\end{cor}

We observe that this nonexistence result depends not only on the behaviour of the kernel at the origin, but on its global behaviour, see~\eqref{mu}. In fact when the kernel is
$$
J(z)=\begin{cases}
  |z|^{-N-\alpha_1}&\text{ if } |z|<1,\\
  |z|^{-N-\alpha_2}&\text{ if } |z|>1,
\end{cases}
$$
$\alpha_1<p$, $\alpha_2>0$, we get $\sigma=\max\{\alpha_1,\alpha_2\}$. It will be interesting to know if only the singularity of $J$ at the origin determines by its own the existence or nonexistence of solution. If this is the case we would get, in the critical singularity exponent $q_*=0$ in~\eqref{kernel0}, that there is no solution for any $m>p$. This, together with the existence result for $m<p$ of Theorem~\ref{th-sublinear}, leaves only the case $m=p$ to be studied. We dedicate next section to this task.

%%%%%%%%%%%%%%%%%%%%%%%%%%%%%%%%%%%%%%%%%%%%%%%%
\section{The eigenvalue problem}\label{sect-eigen}
\setcounter{equation}{0}

In this last section we study the so called eigenvalue problem
\begin{equation}\label{eigenvalueproblem}
\begin{cases}
\mathcal{L}u=\lambda\psi(u),& \mbox{in }\Omega,\\ u=0,& \mbox{in }\Omega^c.
\end{cases}
\end{equation}
The first eigenvalue and eigenfunction are obtained minimizing
\begin{equation*}\label{functional}
I(v)=\frac{E(v)}{F(v)},
\qquad v\in{W}^{J,\Psi}_0(\Omega)\setminus\{0\}.
\end{equation*}
In fact, if $u$ is a minimum, the function $g(t)=I(u+t\varphi)$, for any admisible function $\varphi$ satisfies $g(0)=\frac{E(u)}{F(u)}=\lambda_1$, $g'(0)=0$, that is,
$$
\langle E'(u),\varphi\rangle=\lambda_1\langle F'(u),\varphi\rangle,
$$
which is the associated Euler-Lagrange equation, the weak formulation \eqref{weaksol}.

\begin{theorem}\label{Principal_eigen}
    Define
    $$
    \lambda_1 = \displaystyle \inf_{v\in W^{J,\Psi}_0(\Omega)\setminus\{0\}}I(v).
    $$
    Then $\lambda_1$ is positive and is achieved by some $u\in W^{J,\Psi}_0(\Omega)\setminus\{0\}$.  Moreover, the function $u$ is a weak solution to problem \eqref{eigenvalueproblem}. The solution does not change sign, and it is moreover bounded if~\eqref{psi-plap} and~\eqref{alpha} holds for some $\alpha>0$.
\end{theorem}
\begin{proof}
    The inequality \eqref{poincare1} immediately gives $\lambda_1>0$.
    Consider $\mathcal{M}=\{v\in W^{J,\Psi}_0(\Omega)\,:\,F(v)=1\}$. Let $\{v_n\}$ be a minimizing sequence for $I$ in $\mathcal{M}$, that is
    $$
    \lim_{n\to\infty}I(v_n)=\lambda_1=\inf_{v\in\mathcal{M}}I(v)>0.
    $$
    Then $\{v_n\}$ is bounded in $ W^{J,\Psi}_0(\Omega)$, so there exists a subsequence, still denoted by $\{v_n\}$, such that $v_n\rightharpoonup u$ in $ W^{J,\Psi}_0(\Omega)$. As usual, by Theorem~\ref{th-compact2} there exists a subsequence converging to $u$ in $L^\Psi(\Omega)$, so $F(u)=1$ and $u\in\mathcal{M}$. This gives
    $$
    \lambda_1\le I(u)=E(u)\leq\lim_{n\to\infty}E(v_n)=\lim_{n\to\infty}I(v_n)  =\lambda_1,
    $$
    and then $I(u)= \lambda_1$. The functionals $E$ and $T$ are differentiable, and so is $I$, and we have
    $$
    0=\langle I'(u),\varphi\rangle=\frac{1}{F(\varphi)}\left(\langle E'(u),\varphi\rangle-I(u)\langle F'(u),\varphi\rangle\right).
    $$
    Therefore
    $$
\mathcal{E}(u;\varphi)=\langle E'(u),\varphi\rangle=I(u)\langle F'(u),\varphi\rangle=\lambda_1\displaystyle\int_\Omega \psi(u)\varphi,
    $$
    for every $\varphi \in  W^{J,\Psi}_0(\Omega)$. The fact that the eigenfunction is nonnegative or nonpositive follows by~\eqref{kato-simple} which implies $I(\pm|u|)\le I(u)$. The boundedness of $u$ assuming~\eqref{alpha} is easily proved again by the Moser iterative scheme as performed in \cite{BrascoLindgrenParini}. The key point is the use of the Stroock-Varopoulos inequality \eqref{SV-plap} and condition~\eqref{psi-plap}, and finally apply Theorem~\ref{Linf}. See also~\cite{Franzina-Palatucci}.
\end{proof}
\qed

\section*{Acknowledgments}

Work supported by the Spanish project  MTM2014-53037-P.

\

\noindent{\bf Addresses:}

\noindent{\sc E. Correa: } Departamento de Matem\'{a}ticas, Universidad Carlos III de Madrid, 28911 Legan\'{e}s,
Spain. (e-mail: ernesto.correa@uc3m.es)

\noindent{\sc A. de Pablo: } Departamento de Matem\'{a}ticas, Universidad Carlos III de Madrid, 28911 Legan\'{e}s,
Spain. (e-mail: arturop@math.uc3m.es).

\begin{thebibliography}{B}


\bibitem{Almgren-Lieb}
\newblock F. J.~Almgren Jr. and E. H.~Lieb.
\newblock Symmetric decreasing rearrangement is sometimes continuous.
\newblock \emph{J. Amer. Math. Soc.}, \textbf{2} (1989), 683--773.

\bibitem{Ambrosetti-Rabinowitz}
\newblock A. Ambrosetti and P.~H. Rabinowitz.
\newblock Dual variational methods in critical point theory and
applications.
\newblock \emph{J. Funct. Anal.}, \textbf{14} (1973), 349--381.

\bibitem{BarriosPeralVita}
\newblock B. Barrios, I. Peral and S. Vita.
\newblock Some remarks about the summability of
nonlocal nonlinear problems.
\newblock  \emph{Adv. Nonlinear Anal.}, \textbf{4} (2015), 91--107.

\bibitem{BrandledePablo}
\newblock C.~Br\"andle and A.~de Pablo.
\newblock Nonlocal heat equations: regularizing effect, decay estimates and Nash inequalities.
\newblock  \emph{Comm. Pure Appl. Anal.},  \textbf{17} (2018), 1161--1178.


\bibitem{BrascoLindgrenParini}
\newblock L. Brasco, E. Lindgren and E. Parini.
\newblock The fractional Cheeger problem.
\newblock  \emph{Interfaces Free Bound.},  \textbf{16} (2014), 419--458.

\bibitem{Brezis}
\newblock H. Brezis.
\newblock  Functional analysis, Sobolev spaces and partial differential equations.
\newblock Universitext. Springer, New York, 2011.

\bibitem{BrezisOswald}
\newblock H. Brezis and L. Oswald.
\newblock Remarks on sublinear elliptic equations.
\newblock \emph{Nonlinear Anal. TMA}, \textbf{10} (1986), 55--64.

\bibitem{Caffarelli}
\newblock L.A. Caffarelli.
\newblock  Non-local diffusions, drifts and games,
\newblock in: Nonlinear Partial Differential Equations (Oslo 2010),
\newblock \emph{Abel Symp}, \textbf{7},
\newblock Springer-verlag, Berlin (2012), 37--52.

\bibitem{CoCo}
\newblock A. Córdoba and D. Córdoba.
\newblock  A pointwise estimate for fractionary derivatives with applications to partial
differential equations.
\newblock \emph{Proc. Natl. Acad. Sci. USA}, \textbf{100} (2003) 15316--15317.

\bibitem{CdP}
\newblock E. Correa and A. de Pablo.
\newblock  Nonlocal operators of order near zero.
\newblock \emph{J. Math. Anal. Appl.}, \textbf{461} (2018) 837--867.


\bibitem{hitch}
\newblock E. Di Nezza, G. Palatucci and E. Valdinoci.
\newblock Hitchhiker's guide to the fractional Sobolev spaces.
\newblock Bull. Sci. Math., \textbf{136} (2012),  521--573.

\bibitem{Iannizzotto-etal}
\newblock A. Iannizzotto, S. Liu, K. Perera and M. Squassina.
\newblock Existence result s for fractional $p$--Laplacian problems via Morse theory.
\newblock \emph{Adv. Calc. var.}, \textbf{9} (2016), 101--125.

\bibitem{Frank-Seiringer}
\newblock R. L.~Frank and R.~Seiringer.
\newblock Non-linear ground state representations and sharp Hardy inequalities.
\newblock \emph{J. Funct. Anal.}, \textbf{255} (2008), 3407--3430.

\bibitem{Franzina-Palatucci}
\newblock G. Franzina and G. Palatucci.
\newblock Fractional $p$--eigenvalues.
\newblock \emph{Riv. Mat. Univ. Parma}, \textbf{5} (2014), 315--328.

\bibitem{Friedrichs}
\newblock K. O. Friedrichs.
\newblock On Clarkson's inequalities.
\newblock \emph{Comm. Pure Appl. Math.}, \textbf{23} (1970) 603--607.


\bibitem{Gilbarg-Trudinger}
\newblock D. Gilbarg and N.S. Trudinger.
\newblock Elliptic partial differential equations of second order.
\newblock Springer-Verlag, Classics in Mathematics, Berlin, 2001.



\bibitem{Kato}
\newblock T. Kato.
\newblock Schrödinger operators with singular potentials.
\newblock \emph{Israel J. Math.}, \textbf{13} (1972) 135--148.


\bibitem{Lindgren-Lindqvist} E. Lindgren and P. Lindqvist.
\newblock Fractional eigenvalues.
\newblock \emph{Calc. Var.
Partial Differential Equations}, \textbf{49} (2014), 795--826.

\bibitem{Lindqvist} P. Lindqvist.
\newblock On the equation $\text{div}(\nabla u|^{p-2}\nabla u)+\lambda|u|^{p-2}u=0$.
\newblock \emph{Proc. Amer. Math. Soc.}, \textbf{1} (1990), 157--164.

%\bibitem{PalatucciSavinValdinoci}
%\newblock G. Palatucci, O. Savin and E. Valdinoci.
%\newblock Local and global minimizers for a variational energy involving a fractional
%norm.
%\newblock  \emph{Ann. Mat. Pura Appl.},  \textbf{192} (2013), 673--718.

\bibitem{Rao-Ren}
\newblock M.~M. Rao and Z.~D. Ren.
\newblock Applications of Orlicz spaces.
\newblock Monographs and Textbooks in Pure and Applied Mathematics, 250. Marcel Dekker, Inc., New York, 2002.

\bibitem{RosOtonSerra}
\newblock X. Ros-Oton and J. Serra.
\newblock The Pohozaev identity for the fractional Laplacian.
\newblock \emph{Arch. Ration. Mech. Anal.}, \textbf{213} (2014), 587--628.

\bibitem{Savin-Valdinoci}
\newblock O. Savin and E. Valdinoci.
\newblock Density estimates for a nonlocal
variational model via the Sobolev inequality.
\newblock \emph{SIAM J. Math. Anal.}, \textbf{43} (2011), 2675--2687.

\bibitem{Savin-Valdinoci2}
\newblock O. Savin and E. Valdinoci.
\newblock Density estimates for a variational
model driven by the Gagliardo norm.
\newblock \emph{J. Math. Pures Appl.}, \textbf{101} (2014), 1--26.

\bibitem{Varopoulos}
\newblock N.~T. Varopoulos.
\newblock Hardy-Littlewood theory for semigroups.
\newblock \emph{J. Funct. Anal.}, \textbf{63} (1985), 240--260.


\end{thebibliography}
\end{document}